\documentclass{article}
\usepackage{amssymb,amsmath}
\usepackage{amsfonts}
\usepackage{latexsym}
\newtheorem{Lem}{Lemma}
\newtheorem{The}{Theorem}
\newtheorem{Pro}{Proposition}
\newtheorem{Rem}{Remark}
\newtheorem{Cor}{Corollary}
\newtheorem{Exa}{Example}

\newcommand{\RR}{\mathbb R}
\newcommand{\proof}{{\em Proof. }}
\newcommand{\cqfd}{\mbox{}\nolinebreak\hfill\rule{2mm}{2mm}\medbreak\par}

\title{\bf $\bf{S^1}$-index theory \\ for the Lorentz force equation}

\author{CRISTIAN BEREANU {\small AND} ALEXANDRU P\^IRVUCEANU
}

\date{}

\begin{document}
	\maketitle
	
	{\noindent {\bf Abstract.} In this paper we prove that the $S^1$-invariance of the Poincar\'e action functional associated to the Lorentz force equation gives the existence of multiple critical points which are periodic solutions with a fixed period. To do this, we prove an  abstract multiplicity result which is based upon the Lusternik-Schnirelman method with the $S^1$-index. The corresponding result in the context of the Fadell-Rabinowitz index is proved  in  Ekeland and Lasry (Ann. Math., 112 (1980)).  The  main feature of our abstract result is that it allows us to  consider nonsmooth functionals satisfying only a weak compactness condition well adapted to the Poincar\'e  functional.}
	
	
	\section{Introduction}
	In what follows $\RR^3$ is
	endowed with the Euclidean scalar product $``\cdot"$ and with the
	Euclidean norm $``| \cdot |".$
	Let $T>0$ be a fixed period.
	Let   $V:[0,T]\times \RR^3\to\RR$ and  $W:[0,T]\times\RR^3\to\RR^3$ be two $C^1$-functions.
	In this paper we consider the Lorentz force equation (LFE) with the  electric potential  $V$ and the magnetic potential $W,$ namely 
	$$
	\left(\frac{q'}{\sqrt{1-|q'|^2}}\right)' = E(t,q)+q'\times B(t,q),
	$$
	where
	$$
	E=-\nabla_q V-\frac{\partial W}{\partial t},\qquad B=\mbox{curl}_q\, W
	$$
	are the electric and magnetic fields respectively, and $E+q'\times B$ is the well known Lorentz force. By a {\bf solution} $q$ of the LFE we mean a function
	$q=(q_1,q_2,q_3):[0,T]\to\RR^3$ of class $C^2$ such that $|q'(t)|<1$ for
	all $t\in [0,T]$ and which verifies the equation. In what follows we consider {\bf $T$-periodic solutions}, that is, solutions $q$ such that 
	$$
	q(0)=q(T), \quad q'(0)=q'(T).
	$$ 
	Following Feynman \cite{Fey} (see also \cite{LL}), the above equation  is  the relativistically correct equation of motion for a single charge in the fields $E$ and $B.$ The Lorentz force equation is first introduced by Lorentz \cite{Lor1} and Poincar\'e \cite{Po2}. For more details about the Lorentz force equation, see for example Lorentz's paper \cite{Lor2}. On the other hand, in the Conclusions Section from Damour's paper \cite{Da} we can read that ``one of the most important  advances made by  Poincar\'e \cite{Po2} is the relativistic electron Lagrangian 
	\begin{eqnarray*}
		L_{\mbox{electron}} = -m_{\mbox{electron}} c^2 \sqrt{1 - \frac{\bf{v}^2}{c^2}},
	\end{eqnarray*}
	where
	$$m_{\mbox{electron}} = \frac{4}{3}\frac{E_{\mbox{em}}}{c^2}."$$
	Concerning the Lorentz force equation and the above   Lagrangian introduced by Poincar\'e we can read in \cite{Fey} that:  ``The formula in this case of relativity is the following:
	\begin{eqnarray*}
		S = -m_0c^2\int_{t_1}^{t_2}\sqrt{1 - \frac{v^2}{c^2}}dt - q\int_{t_1}^{t_2}[\phi(x,y,z,t) - {\bf v}\cdot A(x,y,z,t)]dt
	\end{eqnarray*}
	[...] This action function gives the complete theory of relativistic
	motion of a single particle in an electromagnetic field."
	
	Next, it is proved in Theorem 6  \cite{ABT}  that a function $q:[0,T]\to\RR^3$ is a $T$-periodic solution of the Lorentz force equation with the electric potential $V$ and the magnetic potential $W$ if and only if $q$ is a $T$-periodic  Lipschitz function with $||q'||_{\infty}\leq 1$ and $q$ is a solution of the variational inequality 
	\begin{eqnarray*}
		\int_0^T[\sqrt{1-|q'|^2} - \sqrt{1 - |\varphi'|^2}]dt &+&\int_0^T [\mathcal E(t,q,q')- \nabla_qV(t,q)] \cdot (\varphi -q) dt
		\\&+&
		\int_0^T  W(t,q)\cdot (\varphi' - q')dt \geq 0 
	\end{eqnarray*}
	for all $T$-periodic Lipschitz functions $\varphi:[0,T]\to\RR^3$ such that $||\varphi'||_{\infty}\leq 1,$ where $\mathcal E: [0,T]\times\RR^3\times\RR^3\to\RR^3$ is given by
	\begin{equation*}
		\mathcal E(t,q,p) = (p\cdot D_{q_1}W(t,q), p\cdot D_{q_2}W(t,q),p\cdot D_{q_3}W(t,q)).
	\end{equation*}
	This means more or less that $0\in\partial S(q),$ or $q$ is a critical point of the Poincar\'e action functional $\mathcal I_*$ associated to the Lorentz force equation with $T$-periodic boundary conditions (see Theorem \ref{criticsol}).  
	
	In this paper we consider autonomous potentials $V$ and $W,$ i.e. $V:\RR^3\to\RR$ and $W:\RR^3\to\RR^3,$ so both potentials are independent of the time variable. The main remark is that in  this case the action functional $\mathcal I_*$ is $S^1$-invariant.  Our main result about the Lorentz force equation is roughly speaking as follows (see Theorem \ref{theappl}). First, we assume that $V$ is of class $C^2,$ $V(0)=0,$ $V> 0$ and $V'\ne 0$ on $\RR^3\setminus \{0\},$ and at infinity one has that
$$\lim_{|q|\to\infty}V(q) = l^*>0.$$
 Regarding the behaviour of the electric potential $V$ around the origin, we assume that for some positive integer $m$ and for some $\lambda>0$ one has $$V(q)\geq\lambda |q|^2$$ for all $q$ around the origin. Moreover, we only consider electric potentials $V$ such that $V''$ is bounded. 
 For the magnetic potential $W$, we assume that $W$ is of class $C^2$ with $W, W', W''$ bounded on $\RR^3$. Then, we prove that there exists $\Lambda_m>0$ such that if $\lambda$ introduced above is such that $\lambda\geq\Lambda_m,$ then $\mathcal I_*$ has at least $3m$ critical orbits at negative levels which are $2\pi$-periodic solutions of the Lorentz force equation. The constant $\Lambda_m$ quantifies the interaction of the electric potential $V$ with the magnetic potential $W.$ 
	
	To prove that the Poincar\'e action functional has multiple critical orbits at negative levels, we develop a  Lusternik-Schnirelman theory with the $S^1$-index for nonsmooth functionals. First, we prove our abstract result for smooth functionals satisfying only a weak Palais-Smale compactness condition (Theorem \ref{lssmooths1}). In this smooth context our result is a generalization of the classical Lusternik-Schnirelman method with the $S^1$-index or with the Fadell-Rabinowitz index (see, for example, Theorem 1 \cite{EkeLasam}, Theorem 6.1 \cite{MawWil}, Proposition 10 \cite{ekebook}). The main tool in the proof of our Theorem \ref{lssmooths1} is Ghoussoub's location theorem \cite{ghojram, ghobook, ekebook}. To pass from the smooth case (Theorem \ref{lssmooths1}) to the nonsmooth case (Theorem \ref{lsnonsmooths1}), we use the Ekeland-Lasry regularization procedure (Lemma 7 in \cite{EkeLasam}).
	For other results about periodic solutions of the Lorentz force equation via the Lusternik-Schnirelman method applied to the Poincar\'e action functional, see \cite{ABT2, Ber, BDPasnsp, BDPprep}. In the newtonian situation, for results concerning periodic solutions of nonlinear bounded perturbations of the operator $u\mapsto u''$ see for example \cite{ambzelampa}.
	For a good introduction to the Lusternik-Schnirelman theory, see the monographs \cite{ambrosettimalchiodibook}-Chapter 9,  \cite{ekebook}-Chapter 5, \cite{MawWil}-Chapter 6,  \cite{ghobook}-\mbox{Chapter 7}.
	
	The paper is organized as follows. In Section 2 we prove our abstract result (Theorem \ref{lssmooths1}) . In Section 3 we introduce the Poincar\'e action functional on the Hilbert space $H^1_T$ and prove the suitable compactness condition in the context of $H^1_T,$ that is, Proposition \ref{pslfe}. In Section 4 we prove the main result concerning periodic solutions with a fixed period of the Lorentz force equation (\mbox{Theorem \ref{theappl}}) that was described above.
	
	{\bf Acknowledgements.} We are grateful to the anonymous reviewers for their important remarks and suggestions
	which considerably improve the quality of the paper.
	
	\section{$S^1$-invariant nonsmooth functionals}
	In this section $(X, \langle\cdot,\cdot\rangle)$ is a real Hilbert space and $(Y, \|\cdot\|_Y)$ is a real Banach space such that $X\subset Y$ and the canonical injection $i:(X,\|\cdot\|_X)\to 
	(Y,\|\cdot\|_Y)$ is a compact operator. Next, we consider the functional $\mathcal I: X\to (-\infty, \infty]$ satisfying the following hypothesis:
	\begin{center}
		$(H)$ \ $\mathcal I = \Psi + \mathcal F,$ where $\mathcal F\in C^1(X,\RR)$ and $\Psi:X\to (-\infty, +\infty]$ is a convex lower semicontinous functional with a nonempty domain, i.e. $D(\Psi) = \{q\in X: \Psi(q)<+\infty\}\ne \emptyset$.
	\end{center}
	The dual of $X$ is denoted by $X^*$, and for any $q\in X,$ using the Riesz Representation  Theorem, there exists a unique element $\nabla\mathcal F(q)\in X$ such that 
	$$\mathcal F'(q)[\varphi] = \langle\nabla\mathcal F(q),\varphi\rangle \quad \mbox{for all} \ \varphi\in X.$$
	
	\subsection{Critical points and compactness conditions}
	For any $q\in D(\Psi)$, we consider {\bf the subdifferential} of $\Psi$ at $q$ (see \cite{eketembook}) given by
	\begin{eqnarray*}
		\partial\Psi(q) = \{f^*\in X^*: \Psi(\varphi) - \Psi(q)\geq f^*[\varphi - q] \ \mbox{for all} \ \varphi\in D(\Psi)\}.
	\end{eqnarray*}
	It follows that for any $q\in D(\Psi),$ one has 
	\begin{eqnarray*}
		\partial\mathcal I(q) = \mathcal F'(q) + \partial\Psi(q) = \{\mathcal F'(q) + f^*: f^*\in \partial\Psi(q)\}.
	\end{eqnarray*}
	Notice that, by the Riesz Representation Theorem, $\partial\Psi(q)$ and $\partial\mathcal I(q)$ are subsets of $X$ for every $q\in D(\Psi).$
	A point $q\in X$ is {\bf a critical point} of $\mathcal I$ if $q\in D(\Psi)$ and 
	$$0\in \partial\mathcal I(q),$$
	or, equivalently, 
	$$-\mathcal F'(q)\in\partial\Psi(q),$$
	or, equivalently,  the following variational inequality holds:
	$$\Psi(\varphi) - \Psi(q) + \mathcal F'(q)[\varphi - q]\geq 0 \ \mbox{for all} \ \varphi\in D(\Psi).$$
	{\bf A $(PS)$-sequence at the level $c\in\RR$ for $\mathcal I$} (see \cite{szuaihp}) is a sequence $(q_n)\subset D(\Psi)$ such that  $\mathcal I(q_n)\to c$ and there exists a sequence $(\varepsilon_n)\subset [0,\infty)$ having the property that $\varepsilon_n\to 0$ and, for every $n\in \mathbb N,$ one has 
	$$\Psi(\varphi) - \Psi(q_n) + \mathcal F'(q_n)[\varphi - q_n]\geq -\varepsilon_n \|\varphi - q_n\| \ \mbox{for all} \ \varphi\in D(\Psi).$$
	{\bf A $(PS)^*$-sequence at the level $c\in\RR$ for $\mathcal I$} (see \cite{szuaihp}) is a sequence $(q_n)\subset D(\Psi)$ such that  $\mathcal I(q_n)\to c$ and there exists a sequence $(f_n^*)\subset X^*$ having the property that $\|f_n^*\|\to 0$ and, for every $n\in \mathbb N,$ one has 
	$$f_n^*\in\partial\mathcal I(q_n),$$
	or, equivalently,
	$$\Psi(\varphi) - \Psi(q_n) + \mathcal F'(q_n)[\varphi - q_n]\geq f_n^*[\varphi - q_n] \ \mbox{for all} \ \varphi\in D(\Psi).$$
	We need the following result (see Lemma 1.3 in \cite{szuaihp}).
	\begin{Lem}\label{lemps}
		Consider a convex lower semicontinous function $\chi:X\to (-\infty, +\infty]$ with $\chi(0)=0.$ Assume that
		$$\chi(q)\geq -\|q\| \ \mbox{for all} \ q\in X.$$
		Then, there exists $f^*\in X^*$ with $\|f^*\|\leq 1$ and
		$$\chi(q)\geq f^*[q] \ \mbox{for all} \ q\in X.$$
	\end{Lem}
	
	\begin{Lem}\label{echivps}
		Consider a sequence $(q_n)\subset D(\Psi)$ and $c\in\RR.$ We have that $(q_n)$ is a $(PS)$-sequence at the level $c$ for $\mathcal I$ if and only if 
		$(q_n)$ is a $(PS)^*$-sequence at the level $c$ for $\mathcal I.$
	\end{Lem}
	\proof Assume that $(q_n)$ is a $(PS)^*$-sequence at the level $c$ for $\mathcal I,$ i.e. $\mathcal I(q_n)\to c$ and there exists a sequence $(f_n^*)\subset X^*$ having the property that $\|f_n^*\|\to 0$ and, for every $n\in \mathbb N,$ one has
	$$\Psi(\varphi) - \Psi(q_n) + \mathcal F'(q_n)[\varphi - q_n]\geq f_n^*[\varphi - q_n] \ \mbox{for all} \ \varphi\in D(\Psi).$$
	For any $n\in\mathbb N$, one has that
	$$|f_n^*[x]|\leq \|f_n^*\|\|x\| \ \mbox{for all} \ x\in X.$$
	In particular,
	$$f_n^*[\varphi - q_n]\geq -\|f_n^*\| \|\varphi - q_n\| \ \mbox{for all} \ \varphi\in D(\Psi).$$
	It follows that
	$$\Psi(\varphi) - \Psi(q_n) + \mathcal F'(q_n)[\varphi - q_n]\geq -\|f_n^*\| \|\varphi - q_n\| \ \mbox{for all} \ \varphi\in D(\Psi),$$
	and $(q_n)$ is a $(PS)$-sequence at the level $c$ for $\mathcal I$ by just taking $\varepsilon_n=\|f_n^*\|$ $(n\in\mathbb N).$
	
	Next, assume that $(q_n)$ is a $(PS)$-sequence at the level $c$ for $\mathcal I,$ i.e. $\mathcal I(q_n)\to c$ and there exists a sequence
	$(\varepsilon_n)\subset [0,\infty)$ having the property that $\varepsilon_n\to 0$ and, for every $n\in \mathbb N,$ one has 
	$$\Psi(\varphi) - \Psi(q_n) + \mathcal F'(q_n)[\varphi - q_n]\geq -\varepsilon_n \|\varphi - q_n\| \ \mbox{for all} \ \varphi\in D(\Psi).$$
	If $\varepsilon_n=0,$ we take $f_n^*=0.$ Now consider $\varepsilon_n>0$ and  $\chi:X\to (-\infty, +\infty]$ given by
	$$\chi(q)=\varepsilon_n^{-1}(\Psi(q+q_n) - \Psi(q_n) + \mathcal F'(q_n)[q]) \quad (q\in X).$$
	It is clear that $\chi$ is convex, lower semicontinuous, $\chi(0)=0$, and 
	$$\chi(q)\geq -\|q\| \ \mbox{for all} \ q\in X.$$
	Hence, using the above Lemma  \ref{lemps}, there exists $g_n^*\in X^*$ with $\|g_n^*\|\leq 1$ and
	$$\chi(q)\geq g_n^*[q] \ \mbox{for all} \ q\in X.$$
	If we take $f_n^*=\varepsilon_n g_n^*,$ we have $\|f_n^*\|\to 0$ and, using that
	$$\varepsilon_n\chi(\varphi - q_n)\geq f_n^*[\varphi - q_n] \ \mbox{for all} \ \varphi\in D(\Psi),$$
	the conclusion is now clear. \cqfd
	
	\subsection{Ekeland-Lasry regularization procedure}
	The following result, which is one of the main tools in this paper, is due to Ekeland and Lasry, Lemma 7 \cite{EkeLasam}.
	\begin{Lem}\label{ekelas}
		Assume that $\mathcal I$ is bounded from below and the following assumption holds true:
		\begin{center}
			(*) There exists $\alpha>0$ such that the function 
			$D(\Psi)\ni q\mapsto \mathcal I(q) + \alpha\|q\|^2\in (-\infty,\infty)$
			is convex.
		\end{center}
		Consider $0<\varepsilon <\alpha^{-1}$, and let $\mathcal I_{\varepsilon}:X\to \RR$ be given by
		\begin{eqnarray*}
			\mathcal I_{\varepsilon}(q) = \inf_{\varphi\in X}(\varepsilon^{-1}\|\varphi - q\|^2 + \mathcal I(\varphi)).
		\end{eqnarray*}
		The functional $\mathcal I_{\varepsilon}$ satisifes the following properties:
		
		$(EL1)$ $\mathcal I_{\varepsilon}\in C^1(X,\RR);$
		
		$(EL2)$  for all $q\in X,$ one has that
		$$ \inf_{X}\mathcal I \leq \mathcal I_{\varepsilon}(q)\leq \mathcal I(q);$$
		
		$(EL3)$ for a fixed $q\in X,$ one has that
		$$\mathcal I_{\varepsilon}'(q)=0 \Leftrightarrow [q\in D(\Psi), 0\in\partial\mathcal I(q)] \Leftrightarrow \mathcal I_{\varepsilon}(q)=\mathcal I(q).$$
		Moreover, consider the function $\gamma:X\to X$ given by
		$$\gamma(q) = q- \frac{\varepsilon}{2}\nabla\mathcal I_{\varepsilon}(q) \quad (q\in X).$$
		Then, for any $q\in X,$ one has that $\gamma(q)\in D(\Psi)$, and
		
		$(EL4)$ $\mathcal I(\gamma(q)) = \mathcal I_{\varepsilon}(q) - \varepsilon^{-1}\|q - \gamma(q)\|^2;$
		
		$(EL5)$ $2\varepsilon^{-1}(q - \gamma(q)) \in \partial\mathcal I(\gamma(q)).$
	\end{Lem}
	\proof $(EL2)$ is relation $(41)$ in Lemma 7 \cite{EkeLasam}. $(EL3)$ is relation $(42)$ in Lemma 7 \cite{EkeLasam}. Our function $\gamma$ is the function $\psi$ introduced in Step 1 of the proof of Lemma 7 \cite{EkeLasam}. The formula for $\gamma$ given above is relation $(63)$ in Step 2 of the proof of Lemma 7 \cite{EkeLasam}. $(EL4)$ is relation $(47)$ in Step 1 of the proof of Lemma 7 \cite{EkeLasam}. $(EL5)$ is relation $(48)$ in Step 1 of the proof of Lemma 7 \cite{EkeLasam}. \cqfd
	
	Given $c\in\RR$, we say that $\mathcal I$  satisfies {\bf the weak Palais-Smale condition at the level $c$} (for short, {\bf $(wPS)_c$-condition}) (see \cite{ABT, ABT2}) if for any sequence $(q_n)$ from $D(\Psi)$ which is a $(PS)$-sequence at the level $c,$ there exists a subsequence $(q_{n_k})$ converging in $Y$ to a critical point $q^*$ of $\mathcal I$ such that $\mathcal I(q^*)=c,$ that is $q^*$ is a critical point of $\mathcal I$ at the level $c$ and $\|q_{n_k} - q^*\|_Y\to 0.$
	
	Given $c\in\RR$, we say that $\mathcal I_{\varepsilon}$  satisfies {\bf the weak Palais-Smale condition at the level $c$} if for any sequence $(q_n)$ from $X$ such that $\mathcal I_{\varepsilon}(q_n)\to c$ and $\mathcal I_{\varepsilon}'(q_n)\to 0,$ there exists a subsequence $(q_{n_k})$ converging in $Y$ to $q^*\in X$  such that $\mathcal I_{\varepsilon}(q^*)=c$ and $\mathcal I_{\varepsilon}'(q^*)=0.$
	
	\begin{Lem}\label{iepsilonps}
		Consider $c\in\RR$ such that $\mathcal I$ satisfies $(wPS)_c$-condition. If $0<\varepsilon <\alpha^{-1}$, with $\alpha$ given in $(*),$ then $\mathcal I_{\varepsilon}$ satisfies $(wPS)_c$-condition.
	\end{Lem}
	\proof Consider a sequence $(q_n)\subset X$ such that
	$$\mathcal I_{\varepsilon}(q_n)\to c \ \mbox{and} \ \mathcal I_{\varepsilon}'(q_n)\to 0.$$
	For any positive integer $n$, we consider
	$$\varphi_n = \gamma(q_n), \quad u_n = 2\varepsilon^{-1}(q_n - \varphi_n).$$
	Using Lemma \ref{ekelas} - $(EL4)$, we have that
	\begin{eqnarray*}
		\mathcal I(\varphi_n) &=& \mathcal I(\gamma(q_n))\\
		&=& \mathcal I_{\varepsilon}(q_n) - \varepsilon^{-1}\|q_n - \gamma(q_n)\|^2\\
		&=& \mathcal I_{\varepsilon}(q_n) - \varepsilon^{-1}\|q_n - \varphi_n\|^2\\
		&=& \mathcal I_{\varepsilon}(q_n) - \frac{\varepsilon}{4}\|u_n\|^2.
	\end{eqnarray*}
	Using Lemma \ref{ekelas} - $(EL5)$, we have that
	$$2\varepsilon^{-1}(q_n - \gamma(q_n))\in \partial\mathcal I(\gamma(q_n)),$$
	that is
	$$u_n\in \partial\mathcal I(\varphi_n).$$
	But, from the definition of $\gamma$, one has that
	$$\nabla\mathcal I_{\varepsilon}(q_n) = 2\varepsilon^{-1}(q_n - \gamma(q_n)) = u_n,$$
	so $u_n\to 0$ and $\mathcal I(\varphi_n)\to c.$ We deduce that $(\varphi_n)$ is a $(PS)^*$-sequence of $\mathcal I$ at the level $c,$ and using Lemma \ref{echivps}, it follows that $(\varphi_n)$ is a $(PS)$-sequence of $\mathcal I$ at the level $c.$ But $\mathcal I$ satisfies $(wPS)_c$-condition, so there exists a subsequence $(\varphi_{n_k})$ and $q^*\in D(\Psi)$ such that $\mathcal I(q^*) = c,$  $0\in\partial\mathcal I(q^*),$ and $\|\varphi_{n_k} - q^*\|_Y\to 0.$ Then, using Lemma \ref{ekelas} - $(EL3),$ we deduce that $\mathcal I_{\varepsilon}(q^*)=c,$ $\mathcal I_{\varepsilon}'(q^*)=0,$ and $\mathcal I_{\varepsilon}$ satisfies $(wPS)_c$-condition. \cqfd
	
	\subsection{ Benci's $S^1$-index and Ghoussoub's location theorem}
	In this subsection we recall some known results.
	
	First, we introduce Benci's $S^1$-index (see \cite{bencicpam, MawWil}).
	Consider
	$$S^1 = \{z\in\mathbb C: |z|=1\},$$
	which is a compact topological group with the multiplication of complex numbers. We identify the reals $\mbox{mod} \  2\pi,$ that is $\RR/2\pi$, with $S^1$ by $\theta\leftrightarrow e^{i\theta}.$ Thus, the group $(\RR/2\pi, +)$ is identified with $(S^1, \cdot).$ We consider a real Banach space $\mathcal X,$  and we denote by $B(\mathcal X)$ the Banach space of all bounded linear operators acting on $\mathcal X$ endowed with the usual operator norm. We recall that an operator $A\in B(\mathcal X)$ is called an {\bf isometry} if $\|Ax\| = \|x\|$ for all $x\in\mathcal X.$ Next, a {\bf representation} of $S^1$ over the Banach space $\mathcal X$ is a function
	$$L:S^1\to B(\mathcal X): \quad \theta\mapsto L(\theta)$$
	having the following properties:
	\begin{eqnarray*}
		&L(0) = \mbox{id},\\
		&L(\theta_1 + \theta_2) = L(\theta_1)L(\theta_2) \ \mbox{for all} \ \theta_1, \theta_2\in S^1,\\
		&S^1\times\mathcal X\ni (\theta, x)\mapsto L(\theta)x\in\mathcal X \ \mbox{is continuous}.
	\end{eqnarray*}
	A subset $A$ of $\mathcal X$ is {\bf invariant} under the representation $L$ if $L(\theta)A=A$ for all $\theta\in S^1.$ A representation $L$ of $S^1$ over $\mathcal X$ is {\bf isometric} if $L(\theta)$ is an isometry for all $\theta\in S^1.$ A mapping $R$ between two invariant subsets of $\mathcal X$ under the representation $L$ of $S^1$ is {\bf equivariant} if
	$$R\circ L(\theta) = L(\theta)\circ R \ \mbox{for all} \ \theta\in S^1.$$
	Consider 
	$$\mathcal C=\{A\subset \mathcal X: A \ \mbox{is closed and invariant}\}.$$
	The {\bf $S^1$-index} associated to the isometric representation $L$ is the function
	$$\operatorname{ind}:\mathcal C\to \mathbb N\cup \{+\infty\}: \ A\mapsto \operatorname{ind} (A)$$
	defined as follows: for $A\in \mathcal C,$ the $S^1$-index of $A$ is the smallest integer $k$ such that there exists $n\in \mathbb N\setminus \{0\}$ and $\Phi\in C(A, \mathbb C^k\setminus \{0\})$ with 
	$$\Phi(L(\theta)x)=e^{in\theta}\Phi(x) \ \mbox{for all}\  \theta\in S^1, x\in A.$$
	If such a mapping does not exist, then we define
	$$\operatorname{ind}(A) = +\infty.$$
	Finally, we define
	$$\operatorname{ind}(\emptyset) = 0.$$
	The mapping $\mbox{``ind"}$ defined above is an index for the representation $L,$ that is it has the following properties:
	
	(i) $\operatorname{ind}(A) = 0$ if and only if $A=\emptyset;$
	
	(ii) if $R:A_1\to A_2$ is equivariant and continuous, then
	$$\operatorname{ind}(A_1)\leq \operatorname{ind}(A_2);$$
	
	(iii) if $A\in \mathcal C$ is compact, then there exists $N\in\mathcal C$ with $A\subset  \operatorname{int}(N)$ and
	$$\operatorname{ind}(N) = \operatorname{ind}(A);$$
	
	(iv) for all $A_1, A_2\in\mathcal C,$ one has that
	$$\operatorname{ind}(A_1\cup A_2)\leq \operatorname{ind}(A_1) + \operatorname{ind}(A_2).$$ 
	
	Next, we introduce Ghoussoub's location theorem (see \cite{ghojram, ghobook, ekebook}). A functional $\mathcal J:\mathcal X\to (-\infty,\infty]$ is {\bf invariant} for the representation $L$ of $S^1$ over $\mathcal X$ if
	$$\mathcal J\circ L(\theta) = \mathcal J \ \mbox{for all} \ \theta\in S^1.$$
	From Lemma 7 in \cite{EkeLasam}, one has the following result.
	\begin{Lem}\label{invfunct}
		Under the assumptions of Lemma \ref{ekelas},
		if $\mathcal I$ is invariant, then $\mathcal I_{\varepsilon}$ is invariant for any $0<\varepsilon<\alpha^{-1}.$
	\end{Lem}
	A class $\mathcal G$ of compact invariant subsets of $\mathcal X$ is {\bf stable by equivariant homotopies} if for any $A\in\mathcal G$ and $h\in C([0,1]\times \mathcal X,\mathcal X)$ such that $h(s,\cdot)$ is equivariant for all $s\in [0,1]$ and $h(0,\cdot)=\mbox{id},$ one has $h(1,A)\in\mathcal G.$
	
	\begin{The}\label{Ghoussoub}
		Consider a class $\mathcal G$ of compact invariant subsets of $\mathcal X$ which is stable by equivariant homotopies, and let  $\mathcal J:\mathcal X\to\RR$ be a $C^1$ invariant functional such that
		$$c:=\inf_{A\in\mathcal G} \max_A \mathcal J>-\infty.$$
		Assume that $F$ is a closed invariant subset of $\mathcal X$ such that for all $A\in\mathcal G,$ one has that
		\begin{itemize}
			\item $A\cap F\neq\emptyset,$ 
			\item $\sup\limits_{A\cap F}\mathcal J\geq c.$
		\end{itemize}
		Then, there exists a sequence $(x_n)$ in $\mathcal X$ such that
		$$\mathcal J(x_n)\to c, \quad \mathcal J'(x_n)\to 0,$$
		and one has the following localization:
		$$\operatorname{dist}(x_n, F)\to 0.$$
	\end{The}
	
	\subsection{Lusternik-Schnirelman method with the $S^1$-index and weak compactness}
	
	{\bf Computations with the $S^1$-index}
	
	In this subsection we work in the abstract setting we have introduced. Let $L$ be an isometric representation of $S^1$ over the Banach space $Y$ such that $L(\theta)X=X$ for all $\theta\in S^1$, and $L|_X = \{L(\theta)|_{X}\}_{\theta\in S^1}$ is an isometric representation of $S^1$ over the Hilbert space $X$ with the associated (to $L|_X$) $S^1$-index also denoted as above by $``\operatorname{ind}".$ Consider
	$$\operatorname{Fix} (S^1) = \{y\in Y: L(\theta)y=y \ \mbox{for all} \ \theta\in S^1\},$$
	and assume that
	$$\operatorname{Fix}(S^1)\subset X.$$
	For any $y\in Y,$ the {\bf orbit} of $y$ is given by
	$$\mathcal O(y) = \{L(\theta)y: \theta\in S^1\}.$$
	For every $y\in Y$ and $\delta>0$, we consider
	$$\mathcal O^{\delta}(y) = \{z\in Y: \operatorname{dist}_Y(z, \mathcal O(y))\leq \delta\}.$$
	Notice that $\mathcal O(y)$ is compact invariant in $Y$, and $\mathcal O^{\delta}(y)$ is closed invariant in $Y.$ For any $q\in X$, one has that $\mathcal O(q)\subset X,$ but $\mathcal O^{\delta}(q)\subset Y$  $(\delta>0).$
	It is clear that
	$$\operatorname{ind}(\mathcal O(q)) = +\infty \ \mbox{for all} \ q\in \operatorname{Fix}(S^1).$$
	Moreover, if $q\in X\setminus \operatorname{Fix}(S^1),$ then the  $2\pi$-periodic continuous function
	$$S^1\ni \theta\mapsto L(\theta)q\in X$$
	is not constant, so it has a minimal period $T^*>0$, and there exists a positive integer $n$ such that $2\pi = nT^*.$ Thus, for any positive integer $j,$ the function 
	$$\mathcal O(q)\ni L(\theta)q\mapsto e^{ijn\theta}\in\mathbb C\setminus \{0\}, \ \theta\in S^1,$$
	is well-defined and continuous. Hence, in particular, it is clear that
	$$\operatorname{ind}(\mathcal O(q)) = 1.$$
	The next result is a generalization of this remark.
	
	\begin{Lem}\label{indexorbit}
		If $q_1, ... , q_m\in X\setminus \operatorname{Fix}(S^1),$ then there exists $\delta>0$ such that  
		$$\bigcup_{j=1}^m\mathcal O^{\delta}(q_j)\cap \operatorname{Fix}(S^1)=\emptyset, \quad \operatorname{ind}\left(\bigcup_{j=1}^m\mathcal O^{\delta}(q_j)\cap X\right) = 1.$$
	\end{Lem}
	\proof  First, it is clear that there exists $\delta'>0$ such that $\mathcal O^{\delta'}(q_j)\cap \mathcal O^{\delta'}(q_k)=\emptyset$ for all $1\leq j<k\leq m.$
	For $j=1, ... , m$ fixed,
	we consider the $2\pi$-periodic continuous function
	$$S^1\ni \theta\mapsto L(\theta)q_j\in X\subset Y.$$
	Using that $q_j\notin \operatorname{Fix}(S^1),$ the above function is not constant, so it has a minimal period $T_j>0$, and there exists a positive integer $n_j$ such that $2\pi = n_jT_j.$ Consider $n=n_1n_2\cdots n_m$ and the $\|\cdot\|_Y$-continuous function
	$$\Phi:\bigcup_{j=1}^m\mathcal O(q_j)\to\mathbb C\setminus \{0\}\ : \quad L(\theta)q_j\mapsto e^{in\theta}, \ 1\leq j\leq m, \ \theta\in S^1.$$
	Using Tietze's extension theorem, there exists a $\|\cdot\|_Y$-continuous function $\tilde{\Phi}:Y\to\mathbb C$ such that $\tilde{\Phi}|_{\cup_{j=1}^m\mathcal O(q_j)} = \Phi.$ Now we define the continuous function $\Lambda:Y\to\mathbb C$ by
	$$\Lambda (y) = \frac{1}{2\pi}\int_0^{2\pi}e^{-in\theta}\tilde{\Phi}(L(\theta)y)d\theta \ \mbox{for all} \ y\in Y.$$
	For any $\tau\in [0,2\pi]$ and $y\in Y$, one has that
	\begin{eqnarray*}
		\Lambda(L(\tau)y) &=&  \frac{1}{2\pi}\int_0^{2\pi}e^{-in\theta}\tilde{\Phi}(L(\theta)L(\tau)y)d\theta\\
		&=&  \frac{1}{2\pi}\int_0^{2\pi}e^{-in\theta}\tilde{\Phi}(L(\theta+\tau)y)d\theta\\
		&=& \frac{1}{2\pi}\int_{\tau}^{2\pi+\tau}e^{-in(t-\tau)}\tilde{\Phi}(L(t)y)dt\\
		&=& e^{in\tau}\Lambda(y).
	\end{eqnarray*}
	In particular, for $y=q_j$ with $1\leq j\leq m,$ one has that
	\begin{eqnarray*}
		\Lambda(L(\tau)q_j) = e^{in\tau}\Lambda(q_j) = e^{in\tau} = \Phi(L(\tau)q_j),
	\end{eqnarray*}
	that is
	\begin{eqnarray*}
		\Lambda(q) = \Phi(q) \quad\mbox{for all} \ q\in \bigcup_{j=1}^m\mathcal O(q_j).
	\end{eqnarray*}
	We claim that there exists $\delta\leq\delta'$ such that $\Lambda(q)\neq 0$ for all $q\in \cup_{j=1}^m\mathcal O^{\delta}(q_j)\cap X.$ Assume by contradiction that our claim is not true. Then, for all positive integers $k\geq \delta'^{-1}$, there exists $v_k\in \cup_{j=1}^m\mathcal O^{1/k}(q_j)\cap X$ with $\Lambda(v_k)=0.$ Let $u_k\in \cup_{j=1}^m\mathcal O(q_j)$ be such that $\|v_k - u_k\|_Y\leq 1/k.$ Passing to a subsequence, there exists $u\in  \cup_{j=1}^m\mathcal O(q_j)$ such that $\|u_k - u\|_X\to 0,$ and  so  $\|v_k - u\|_Y\to 0.$ Hence, $\Lambda(v_k)\to\Lambda(u)$ and $\Lambda(u)=\Phi(u)=0,$ contradicting the definition of $\Phi.$ Thus, $\Lambda:  \cup_{j=1}^m\mathcal O^{\delta}(q_j)\cap X\to\mathbb C\setminus \{0\}$ is a $\|\cdot\|_X$-continuous function such that $\Lambda(L(\theta)q)=e^{in\theta}\Lambda(q)$ for all $q\in \cup_{j=1}^m\mathcal O^{\delta}(q_j)\cap X,$ $\theta\in S^1,$ which implies that $\operatorname{ind}(\cup_{j=1}^m\mathcal O^{\delta}(q_j)\cap X) = 1.$ Now the other conclusion is clear. \cqfd

	{\bf The smooth case}
	
	Next, consider an invariant functional $\mathcal J\in C^1(X,\RR)$, and for any positive integer $j$ we define
	$$\mathcal G_j = \{A\subset X: A \ \mbox{compact, invariant with} \ \operatorname{ind}(A)\geq j\}$$
	and the associated Lusternik-Schnirelman levels,
	$$c_j = \inf_{A\in\mathcal G_j}\max_A \mathcal J.$$
	Using property (ii) of the $S^1$-index, it follows that the class $\mathcal G_j$ is stable by equivariant homotopies for any positive integer $j.$
	Notice that $\operatorname{ind}(\mathcal O(0))=+\infty,$ so $\mathcal O(0)\in\mathcal G_j$ and $c_j\leq \mathcal J(0)$ for all positive integers $j.$ 
	Hence, it follows that
	$$-\infty\leq c_1\leq c_2\leq \cdots \leq\mathcal J(0).$$
	For $c\in\RR$, we consider the set of critical points of $\mathcal J$ at the level $c,$
	$$K_c = \{q\in X: \mathcal J(q)=c, \mathcal J'(q)=0\}.$$
	Notice that since $\mathcal J$ is invariant, it follows that if $q\in K_c,$ then $\mathcal O(q)\subset K_c.$ We say that $\mathcal O(q)$ is a {\bf critical orbit} of $\mathcal J$ at the level $c.$
	
	\begin{The}\label{lssmooths1}
		Let $\mathcal J\in C^1(X,\RR)$ be an invariant functional.
		\begin{itemize}
			\item  If for some positive integer $j$ one has that $c_j>-\infty$ and $\mathcal J$ satisfies $(wPS)_{c_j}$-condition, then $K_{c_j}\neq\emptyset.$
			\item  Moreover, if for some positive integers $j<k$ one has that $c_j=c_k=c>-\infty,$  $\mathcal J$ satisfies $(wPS)_{c}$-condition, and $K_c\cap\operatorname{Fix}(S^1)=\emptyset,$ then there exists a sequence $(q_n)\subset X$ such that $\cup_n\mathcal O(q_n)\subset K_c$ and $\mathcal O(q_m)\cap\mathcal O(q_n)=\emptyset$ for all positve integers $m, n$ such that  $m\neq n,$ that is, there exists a sequence of distinct critical orbits of $\mathcal J$ at the level $c.$
		\end{itemize}
	\end{The}
	\proof   We consider a positive integer $j$ such that $c_j>-\infty$ and $\mathcal J$ satisfies $(wPS)_{c_j}$-condition. Using Ghoussoub's location theorem, that is Theorem \ref{Ghoussoub}, with $F=X$ and $\mathcal G = \mathcal G_j,$ it follows that there exists a sequence $(q_n)\subset X$ such that 
	$$\mathcal J(q_n)\to c_j, \quad \mathcal J'(q_n)\to 0.$$
	Using that $\mathcal J$ satisfies $(wPS)_{c_j}$-condition, passing to a subsequence it follows that there exists $q\in X$ such that $\|q_n - q\|_Y\to 0,$ $\mathcal J(q)=c_j$, and $\mathcal J'(q)=0.$ In particular, $q\in K_{c_j},$ so  $K_{c_j}\neq\emptyset.$
	
	For the second assertion, assume by contradiction that  $K_c$ contains only a finite number of orbits $\mathcal O(q_1), \cdots , \mathcal O(q_l).$ Now let $\delta>0$ be given by Lemma \ref{indexorbit}, so one has that  
	$$\bigcup_{m=1}^l\mathcal O^{\delta}(q_m)\cap \operatorname{Fix}(S^1)=\emptyset, \quad \operatorname{ind}\left(\bigcup_{m=1}^l\mathcal O^{\delta}(q_m)\cap X\right) = 1.$$
	Consider in $X$ the invariant set $F$ given by
	$$F=\{q\in X: \ \operatorname{dist}_Y(q,\cup_{m=1}^l\mathcal O(q_m))\geq\delta\}.$$
	Using that the canonical injection from $X$ into $Y$ is a compact operator and 
	$$F= X\cap \ (\operatorname{dist}_Y(\cdot,\cup_{m=1}^l\mathcal O(q_m)))^{-1}[\delta,\infty),$$
	it follows that $F$ is closed in $X.$
	Notice that by our assumption,
	$$F\cap K_c=\emptyset.$$
	For any fixed $A\in\mathcal G_k$, it is clear that
	$$A\subset (A\cap F) \cup (\cup_{m=1}^l\mathcal O^{\delta}(q_m)\cap X),$$
	and using property (iv) of the $S^1$-index we deduce that
	\begin{eqnarray*}
		j<k\leq  \operatorname{ind}(A) \leq \operatorname{ind}(A\cap F) + \operatorname{ind}(\cup_{m=1}^l\mathcal O^{\delta}(q_m)\cap X) =\operatorname{ind}(A\cap F) +1,
	\end{eqnarray*}
	which implies that
	$$\operatorname{ind}(A\cap F)\geq j,$$
	so $A\cap F\neq\emptyset,$ $A\cap F\in \mathcal G_j$, and $\max\limits_{A\cap F}\mathcal J\geq c_j=c.$
	Using Ghoussoub's location theorem, that is Theorem \ref{Ghoussoub}, with $F$ defined above and $\mathcal G = \mathcal G_k,$ it follows that there exists a sequence $(u_n)\subset X$ such that 
	$$\mathcal J(u_n)\to c, \quad \mathcal J'(u_n)\to 0, \ \mbox{and} \ \operatorname{dist}_X(u_n,F)\to 0.$$
	Using that $\mathcal J$ satisfies $(wPS)_{c}$-condition, passing to a subsequence it follows that there exists $u\in X$ such that $\|u_n - u\|_Y\to 0,$ $\mathcal J(u)=c$, and $\mathcal J'(u)=0.$ Let $(v_n)\subset F$ be a sequence in $F$ such that $\|u_n - v_n\|_X\to 0,$ hence $\|v_n - u\|_Y\to 0.$ Using that $\operatorname{dist}_Y(\cdot,\cup_{m=1}^l\mathcal O(q_m))$ is a continuous function on the Banach space $(Y,\|\cdot\|_Y)$ and $u\in X,$ it follows that $u\in F$ and $\mathcal O(u)\subset F\cap K_c,$ contradicting $F\cap K_c=\emptyset.$ Thus, the second assertion is proved. \cqfd
	
	\begin{Rem}
		{\em The above result, that is Theorem \ref{lssmooths1}, holds true without any change in the proof if $X$ is a Banach space.  We need $X$ to be a Hilbert space in order to use the Ekeland-Lasry result (Lemma \ref{ekelas}).}
	\end{Rem}
	
	{\bf The nonsmooth case}
	
	Now let us return to our nonsmooth functional $\mathcal I=\Psi+\mathcal F$ which satisfies the main hypothesis $(H)$, and  assume that the functionals $\Psi$ and $\mathcal F$ are invariant (under the isometric action $L|_{X}$ of the compact group $S^1$ over the Hilbert space $X$). For a fixed real number $c\in\RR,$ we consider, as in the smooth case,
	$$K_c(\mathcal I) = \{q\in X: \mathcal I(q)=c, q \ \mbox {is a critical point of} \ \mathcal I\}.$$
	Moreover, we consider
	$$K(\mathcal I) = \{q\in X:  q \ \mbox {is a critical point of} \ \mathcal I\}.$$
	Using that $\Psi$ and $\mathcal F$ are invariant, it is easy to prove that if $q\in K_{c}(\mathcal I),$ then $\mathcal O(q)\subset K_{c}(\mathcal I).$ As in the smooth case, for any positive integer $j$, we consider the Lusternik-Schnirelman levels associated to the nonsmooth functional $\mathcal I,$
	$$c_j(\mathcal I) = \inf_{A\in\mathcal G_j}\sup_A \mathcal I.$$
	One has that
	$$-\infty\leq c_1(\mathcal I)\leq c_2(\mathcal I)\leq \cdots \leq\mathcal I(0).$$
	Moreover, following Ekeland-Lasry \cite{EkeLasam}, we consider the set 
	$$\Omega = \{q\in X: \mathcal I(q)<0\}.$$
	Our  main abstract  result, which will be applied to the Poincar\'e action \mbox{functional}, goes as follows.
	
	\begin{The}\label{lsnonsmooths1}
		Assume that the invariant nonsmooth functional $\mathcal I$ is bounded from below, satisfies $(H),$ the Ekeland-Lasry condition $(*),$  $\mathcal I(0)=0,$ there exists some $\omega<0$ such that $\mathcal I$ satisfies $(wPS)_d$-condition for all $d<\omega,$ and $$\operatorname{Fix}(S^1) \cap\Omega\cap K(\mathcal I) = \emptyset.$$
		Consider $0<\varepsilon<\alpha^{-1}.$
		\begin{itemize}
			\item  One has that  $K_{c_j(\mathcal I_{\varepsilon})}(\mathcal I_{\varepsilon}) = K_{c_j(\mathcal I_{\varepsilon})}(\mathcal I)$ is a nonempty set for all positive integers $j$ such that $c_j(\mathcal I_{\varepsilon})<\omega.$
			\item Moreover, if 
			there exist   positive integers $j, k$ such that 
			\begin{eqnarray*}
				j<k \ \mbox{and} \ c_j(\mathcal I_{\varepsilon}) = c_k(\mathcal I_{\varepsilon})<\omega,
			\end{eqnarray*}
			then  $ K_{c_j(\mathcal I_{\varepsilon})}(\mathcal I)$ contains infinitely many orbits. 
			\item In particular, if for some positive integer $k$ one has that 
			$
			c_k(\mathcal I_{\varepsilon})<\omega,
			$
			then  $ K(\mathcal I)$ contains at least $k$ orbits at negative levels.
		\end{itemize}
	\end{The}
	\proof From Lemma \ref{ekelas} - $(EL1)$ it follows that $\mathcal I_{\varepsilon}\in C^1(X,\RR),$ and using Lemma 7 in \cite{EkeLasam}, we have that $\mathcal I_{\varepsilon}$ is invariant. Moreover, using Lemma \ref{iepsilonps} and Lemma \ref{ekelas} - $(EL2)$ we deduce that $\mathcal I_{\varepsilon}$ satisfies $(wPS)_d$-condition for all $d<\omega.$ Now the first and second parts of the proof follow immediately from Theorem \ref{lssmooths1} and Lemma \ref{ekelas} - $(EL3).$ The last part follows  from the two previous parts. \cqfd
	We recall the following result (see Proposition 5.3 in \cite{MawWil}).
	\begin{Lem}\label{computindexmw}
		Let $Z$ be a finite-dimensional invariant subspace of $X$ and let $D$ be an open bounded invariant neighbourhood of $0$ in $Z.$ If
		$Z\cap \operatorname{Fix}(S^1) = \{0\},$
		then
		$$\operatorname{ind}(\partial D) = \frac{1}{2}\mbox{dim}(Z).$$
	\end{Lem}
	From the above Lemma \ref{computindexmw} and our main abtract result, Theorem \ref{lsnonsmooths1}, we deduce the following important result.
	\begin{Cor}\label{maincorabstract}
		Assume that the invariant nonsmooth functional $\mathcal I$ is bounded from below and satisfies $(H),$ the  Ekeland-Lasry condition $(*),$ $\mathcal I(0)=0,$ there exists $\omega<0$ such that $\mathcal I$ satisfies $(wPS)_d$ for all $d<\omega$, and $\operatorname{Fix}(S^1) \cap\Omega \cap K(\mathcal I)= \emptyset.$ Let $Z$ be a (2k)-dimensional invariant subspace of $X$ such that $Z\cap \operatorname{Fix}(S^1) = \{0\},$ and let $D$ be an open bounded invariant neighborhood of $0$ in $Z$ with $\sup_{\partial D}\mathcal I<\omega.$ Then, $ K(\mathcal I)$ contains at least $k$ orbits at negative levels.
	\end{Cor}
	\proof Using Lemma \ref{computindexmw}, we deduce that $\operatorname{ind}(\partial D)=k.$ But $\sup_{\partial D}\mathcal I<\omega,$  so using Lemma \ref{ekelas} - $(EL2),$ one has that $\max\limits_{\partial D}\mathcal I_{\varepsilon}<\omega,$ which together with $\partial D\in\mathcal G_k$ implies that $c_k(\mathcal I_{\varepsilon})<\omega.$ Now the result follows from the last part of \mbox{Theorem \ref{lsnonsmooths1}.} \cqfd

	\section{Poincar\'e action functional on $H^1_T$}
	\subsection{Function spaces}
	We denote by $C_T$ the Banach space of continuous functions  $q:[0,T]\to\RR^3$ with $q(0)=q(T),$ endowed with the usual norm
	$$||q||_{\infty}=\max_{[0,T]}|q| \quad\mbox{for all} \ q\in C_T.$$
	The norm in $L^{\infty}(0, T)$ will also be denoted by $||{\cdot}||_{\infty}$. If $W^{1,\infty}(0,T)$ denotes the space of all real valued Lipschitz functions in $[0,T]$ (or, equivalently, the absolutely continuous functions on $[0,T]$ with bounded derivatives a.e.), we consider the Banach space
	\begin{eqnarray*}
		W^{1,\infty}_T=\{ q\in [W^{1,\infty}(0,T)]^3: q(0)=q(T)\}
	\end{eqnarray*}
	endowed with the usual norm $\|\cdot\|_{1,\infty}$ given by
	$$
	\|q\|_{1,\infty}=\|q\|_{\infty} + \|q'\|_{\infty}  \quad (q\in W^{1,\infty}_T).
	$$
	The Sobolev space $H^1_T$ is the space of functions $q\in L^2(0,T;\RR^3)$ having a weak derivative $q'\in  L^2(0,T;\RR^3)$ and which are $T$-periodic, that is $q(0)=q(T),$ or, equivalently, the space of absolutely continuous functions $q:[0,T]\to\RR^3$ with $q'\in  L^2(0,T;\RR^3)$ and which are $T$-periodic. The Sobolev space $H^1_T$ is a Hilbert space with the inner product
	$$(q|\varphi)_{1,2}=\int_0^T[q(t)\cdot\varphi(t) + q'(t)\cdot\varphi'(t)]dt,$$
	and we shall denote the corresponding norm by $\|\cdot\|_{1,2}.$ Notice that $W^{1,\infty}_T\subset H^1_T$ continuously and the Arzel\`a-Ascoli theorem implies that $H^1_T\subset C_T$ compactly.
	
	\subsection{The nonsmooth part of the action functional}
	Consider
	\begin{eqnarray*}
		D(\Psi_*)= \{q\in W^{1,\infty}_T: \|q'\|_{\infty}\leq 1\},
	\end{eqnarray*}
	and  $\Psi_*: H^{1}_T\to (-\infty,+\infty]$ given   by
	\[
	\Psi_*(q)=\left\{ \begin{array}{ll}
		\displaystyle \int_0^T[1-\sqrt{1-|q'(t)|^2}]dt , &\mbox{if }  q\in D(\Psi_*)
		,
		\\
		\\
		+\infty , &\mbox{if } \ q\in  H^{1}_T \setminus D(\Psi_*).
	\end{array}
	\right.
	\]
	Following \cite{mawsurvey, bjmrlma} one has the following result:
	
	\begin{Lem}  \label{propkpsi}
		\begin{enumerate}
			\item[(i)] The set $D(\Psi_*)$ is convex and closed in $C_T.$ Moreover, if $(q_n)$  is a sequence in $D(\Psi_*)$ converging pointwise in $[0,T]$ to a continuous function $q:[0,T]\to\RR^3$,  then $q\in D(\Psi_*)$ and $q'_n\to q'$ in the $w^*$-topology $\sigma(L^{\infty},L^1).$
			\item[(ii)] If $(q_n)$  is a sequence in $D(\Psi_*)$ converging in $C_T$ to $q$, then $q\in\mathcal D(\Psi_*)$ and
			$$
			\Psi_*(q)\leq \liminf_{n\to\infty}\Psi_*(q_n).
			$$
			In particular, the functional $\Psi_*$ is weakly lower semicontinuous and convex on $H^{1}_T$. 
		\end{enumerate}
	\end{Lem}
	
	\subsection{The smooth  part of the action functional}
	Let  $\mathcal F_*:H^1_T\to\RR$ be given by
	\begin{eqnarray*}
		\mathcal F_*(q):=  \int_0^T[q'(t)\cdot W(t,q(t)) - V(t,q(t))]dt \qquad \mbox{for all} \ q\in H^1_T.
	\end{eqnarray*}
	It is standard  to prove (see \cite{MawWil}) that if $V: [0,T]\times \RR^3\to\RR$ and $W: [0,T]\times\RR^3\to\RR^3$ are $C^1$ functions with $W(0, \cdot)=W(T, \cdot)$, then
	$\mathcal F_*\in C^1(H^1_T, \RR)$, with
	\begin{eqnarray*}
		\mathcal F_*'(q)[\varphi] &=&
		\int_0^T(\mathcal E(t,q(t),q'(t))- \nabla_q V(t,q(t)))\cdot \varphi(t) dt\\  &+&
		\int_0^T  W(t,q(t))\cdot \varphi'(t) dt
	\end{eqnarray*}
	for every  $q, \varphi\in H^{1}_T,$ where $\mathcal E: [0,T]\times\RR^3\times\RR^3\to\RR^3$ is given by
	\begin{equation*}
		\mathcal E(t,q,p) = (p\cdot D_{q_1}W(t,q), p\cdot D_{q_2}W(t,q),p\cdot D_{q_3}W(t,q)).
	\end{equation*}
	
	\subsection{The action functional}
	The action functional on $H^1_T$ (see \cite{ABT} for $W^{1,\infty}_T$ case) associated to the LFE with the electric potential $V,$  the magnetic potential $W$, and periodic boundary conditions on $[0,T]$ is given by
	$$\mathcal I_*:H^1_T\to (-\infty, +\infty], \quad \mathcal I_* = \Psi_* + \mathcal F_*.$$
	One has that $\mathcal I_*$ satisfies $(H),$ hence a point $q\in H^1_T$ is {\bf a critical point} of $\mathcal I_*$ if $q\in D(\Psi_*)$ and 
	\begin{eqnarray*}
		\Psi_*(\varphi) - \Psi_*(q) + \mathcal F'_*(q)[\varphi - q]\geq 0 \quad \mbox{for all} \ \varphi\in D(\Psi_*),
	\end{eqnarray*}
	or, equivalently,
	\begin{multline*}
		\int_0^T[\sqrt{1-|q'|^2} - \sqrt{1 - |\varphi'|^2}]dt +\int_0^T [\mathcal E(t,q,q')- \nabla_qV(t,q)] \cdot (\varphi -q) dt
		\\+
		\int_0^T  W(t,q)\cdot (\varphi' - q')dt \geq 0 \quad\mbox{for all} \ \varphi\in D(\Psi_*).
	\end{multline*}
	We need the following important result (see Lemma 14 in \cite{mawsurvey}). For $q\in L^1(0,T;\RR^3)$, we denote 
	$$\overline{q}=\frac{1}{T}\int_0^Tq(t)dt.$$
	\begin{Lem}\label{bremawlem}
		For every $f\in L^1(0,T;\RR^3)$, there exists a unique  $q_f\in W^{2,1}(0,T;\RR^3)$ such that $\|q'_f\|_{\infty}<1$  and 
		\begin{eqnarray*}
			\left(\frac{q'_f}{\sqrt{1-|q'_f|^2}}\right)'  = \overline{q_f} + f, \quad q_f(0)=q_f(T), \ q'_f(0)=q'_f(T).
		\end{eqnarray*}
		Moreover, $q_f$ is the unique solution $q\in D(\Psi_*)$ of the variational inequality
		\begin{eqnarray*}
			\int_0^T[\sqrt{1-|q'|^2} - \sqrt{1 - |\varphi'|^2}]dt + T\overline{q}\cdot (\overline{\varphi} - \overline{q}) + \int_0^T f\cdot (\varphi - q)dt\geq 0\\ \mbox{for all} \ \varphi\in D(\Psi_*).
		\end{eqnarray*}
	\end{Lem}
	
	Using Lemma \ref{bremawlem} and the same strategy as in the proof of Proposition 1 in \cite{bjmrlma} (see also Proposition 1 in \cite{mawsurvey} and Theorem 6 in \cite{ABT}), we have the following key result.
	
	\begin{The}\label{criticsol}
		A function $q\in H^1_T$ is a $T$-periodic solution of the LFE with the electric potential $V$ and the magnetic potential $W$ if and only if $q$ is a critical point of the action functional $\mathcal I_*.$
	\end{The}
	
	In the following result we introduce the weak compactness property of the Poincar\'e action functional on $H^1_T.$ For the $W^{1,\infty}_T$ situation see Lemma 6 in \cite{ABT}.  
	
	\begin{Pro}\label{pslfe}
		Assume that any $(PS)$-sequence $(q_n)$ of $\mathcal I_*$ at the level $c\in\RR$ is such that $(\overline{q_n})$ is bounded. Then, $\mathcal I_*$ satisfies $(wPS)_c$-condition with respect to $C_T.$
	\end{Pro}
	\proof Let $(q_n)\subset D(\Psi_*)$ be a $(PS)$-sequence of $\mathcal I_*$ at the level $c\in\RR.$  Using that $(\overline{q_n})$ is bounded, it follows that $(q_n)$ is bounded in $W^{1,\infty}_T$ and passing to a subsequence we can assume that $\|q-q_n\|_{\infty}\to 0$ as $n\to\infty$ for some $q\in C_T.$ From Lemma \ref{propkpsi} we deduce that $q\in D(\Psi_*),$ $\Psi_*(q)\leq \liminf\limits_{n\to\infty}\Psi_*(q_n)$, and $q'_n\to q'$ in the $w^*$-topology $\sigma(L^{\infty},L^1).$ It follows that
	$$\lim_{n\to\infty}\int_0^TV(t,q_n)dt=\int_0^TV(t,q)dt.$$
	Moreover, for any positive integer $n$, we have 
	\begin{eqnarray*}
		\left|\int_0^Tq'_n\cdot (W(t,q_n) - W(t,q))dt\right| &\leq& 
		\int_0^T|q'_n| \ |W(t,q_n) - W(t,q)|dt \\ &\leq& 
		\int_0^T |W(t,q_n) - W(t,q)|dt,
	\end{eqnarray*}
	which implies that
	$$\lim_{n\to\infty}\int_0^Tq'_n\cdot (W(t,q_n) - W(t,q))dt=0.$$
	Next, using that $W(\cdot,q)\in L^{\infty}(0,T;\RR^3)$ we deduce that
	$$\lim_{n\to\infty}\int_0^Tq'_n\cdot W(t,q)dt = \int_0^T q'\cdot W(t,q)dt,$$
	hence
	$$\lim_{n\to\infty}\int_0^Tq'_n\cdot W(t,q_n)dt = \int_0^Tq'\cdot W(t,q)dt$$
	and
	$$\lim_{n\to\infty}\mathcal F_*(q_n) = \mathcal F_*(q).$$
	Analogously, one has  
	\begin{eqnarray*}
		\lim_{n\to\infty}\int_0^T(\mathcal E(t,q_n,q'_n) - \nabla_qV(t,q_n))\cdot (\varphi - q_n)dt =\\
		\int_0^T(\mathcal E(t,q,q') - \nabla_qV(t,q))\cdot (\varphi - q)dt \quad \mbox{for all} \ \varphi\in D(\Psi_*),
	\end{eqnarray*}
	and we deduce that
	$$\lim_{n\to\infty}\mathcal F'_*(q_n)[\varphi - q_n] = \mathcal F'_*(q)[\varphi - q] \quad\mbox{for all} \ \varphi\in D(\Psi_*).$$
	We recall that $\mathcal I_*(q_n) \to c$ and we consider a sequence $(\varepsilon_n)\subset [0,\infty)$ having the property that  $\varepsilon_n\to 0$ and, for every positive integer $n$, one has 
	$$\Psi_*(\varphi) - \Psi_*(q_n) + \mathcal F'_*(q_n)[\varphi - q_n]\geq -\varepsilon_n \|\varphi - q_n\|_{1,2} \ \mbox{for all} \ \varphi\in D(\Psi_*).$$
	If we take $n\to\infty$, we deduce that $q$ is a critical point of $\mathcal I_*.$ Moreover, taking $\varphi=q$ it follows that
	$$\Psi_*(q) - \Psi_*(q_n) + \mathcal F'_*(q_n)[q - q_n]\geq -\varepsilon_n \|q - q_n\|_{1,2} \ \mbox{for all positive integers} \ n,$$
	and taking $n\to \infty$ it follows that $\Psi_*(q)=\lim\limits_{n\to \infty}\Psi_*(q_n),$ so $c=\mathcal I_*(q).$ The proof is completed. \cqfd
	
	\section{Main result}
	For convenience, we take our fixed period to be $T =2\pi.$ Let $L:S^1\to B(C_{2\pi})$ be given by
	$$(L(\theta)q)(t) = q(t+\theta) \quad \mbox{for all} \ \theta\in S^1, t\in\RR, q\in C_{2\pi}.$$
	One has that $L$ is an isometric representation of $S^1$ over the Banach space $C_{2\pi}.$ 
	Notice that 
	$$\operatorname{Fix}(S^1) = \RR^3\subset H^1_{2\pi}.$$
	Moreover, it is clear that $(L(\theta)|_{H^1_{2\pi}})_{\theta\in S^1}$ is an isometric representation of $S^1$ over the Hilbert space $H^1_{2\pi}.$ The nonsmooth part of the Poincar\'e action $\mathcal I_*,$ that is $\Psi_*,$ is invariant under the representation $L.$
	In this section we consider the Lorentz force equation with autonomous electric and magnetic potentials, so 
	$$V:\RR^3\to\RR, \quad W:\RR^3\to\RR^3.$$
	This implies that the smooth part of the Poincar\'e action $\mathcal I_*,$ that is $\mathcal F_*$, is invariant under the representation $L.$
	
	\begin{Lem}\label{ELconvlem}
		Assume that the conditions 
		\begin{itemize}
			\item   $(V_1)$ $V$ is of class $C^2$ on $\RR^3,$ $V(0)=0,$ $V> 0$ on $\RR^3\setminus \{0\},$ $V'\neq 0$ on $\RR^3\setminus \{0\}$, $V''$ is bounded on $\RR^3$, and there exists $l^*>0$ such that $\lim_{|q|\to\infty}V(q) = l^*.$
			\item $(W_1)$ \ $W$ is of class $C^2$ on $\RR^3$ and $W, W', W''$ are bounded on $\RR^3.$
		\end{itemize}
		are satisfied.
		Then, the action $\mathcal I_*$ satisfies the Ekeland-Lasry convexity condition $(*),$ that is, there exists $\alpha>0$ such that the function
		$$D(\Psi_*)\ni q\mapsto \mathcal I_*(q) + \alpha\|q\|_{1,2}^2\in \RR$$
		is convex.
	\end{Lem}
	\proof Using $(V_1)$, let
	$\alpha_1>0$ be such that the function 
	$$\RR^3\ni q\mapsto -V(q)+\alpha_1|q|^2\in\RR$$
	is convex. But, for any $q\in H^1_{2\pi}$, one has that
	\begin{eqnarray*}
		&-\int_0^{2\pi}V(q(t))dt+\alpha_1\|q\|_{1,2}^2=\int_0^{2\pi}(-V(q(t)))dt\\
		&+\alpha_1\int_0^{2\pi}(|q(t)|^2 + |q'(t)|^2)dt=\int_0^{2\pi}(-V(q(t))+\alpha_1|q(t)|^2)dt\\
		&+\alpha_1\int_0^{2\pi}|q'(t)|^2dt,
	\end{eqnarray*}
	and then we deduce that 
	$$H^1_{2\pi}\ni q\mapsto -\int_0^{2\pi}V(q(t))dt+\alpha_1\|q\|_{1,2}^2\in\RR$$
	is convex. Next, for a fixed  constant $\alpha_2>0$, we consider the function $\mathcal H_*:W^{1,\infty}_{2\pi}\to\RR$ given by
	$$\mathcal H_*(q) = \int_0^{2\pi}q'(t)\cdot W(q(t))dt + \alpha_2\|q\|_{1,2}^2 \quad\mbox{for all} \ q\in W^{1,\infty}_{2\pi}.$$
	One has that $\mathcal H_*\in C^2(W^{1,\infty}_{2\pi},\RR)$, with the second order derivative given by
	\begin{eqnarray*}
		\mathcal H''_*(q)[\varphi,\psi] &=& \int_0^{2\pi}q'(t)\cdot W''(q(t))[\varphi(t),\psi(t)]dt\\
		&+&\int_0^{2\pi}\psi'(t)\cdot W'(q(t))[\varphi(t)]dt\\
		&+&\int_0^{2\pi}\varphi'(t)\cdot W'(q(t))[\psi(t)]dt\\
		&+&2\alpha_2(\varphi|\psi)_{1,2} \ \mbox{for all} \ q, \varphi, \psi\in W^{1,\infty}_{2\pi},
	\end{eqnarray*}
	which implies that
	\begin{eqnarray*}
		\mathcal H''_*(q)[\varphi - q,\varphi - q] &=& \int_0^{2\pi}q'(t)\cdot W''(q(t))[\varphi(t) - q(t),\varphi(t) - q(t)]dt\\
		&+&2\int_0^{2\pi}(\varphi'(t) - q'(t))\cdot W'(q(t))[\varphi(t) - q(t)]dt\\
		&+&2\alpha_2\|\varphi - q\|^2_{1,2} \ \mbox{for all} \ q, \varphi\in W^{1,\infty}_{2\pi}.
	\end{eqnarray*}
	Using $(W_1)$, we consider two positive constants $c_1$ and $c_2$ such that $||W'(q)||\leq c_1$ and $||W''(q)||\leq c_2$ for all $q\in\RR^3.$ Then, for any $q,\varphi\in D(\Psi_*)$, one has
	$$\left|\int_0^{2\pi}q'(t)\cdot W''(q(t))[\varphi(t) - q(t),\varphi(t) - q(t)]dt\right|\leq c_2\|\varphi - q\|_{L^2}^2\leq c_2\|\varphi - q\|^2_{1,2},$$
	and, on the other hand
	\begin{eqnarray*}
		\left|\int_0^{2\pi}(\varphi'(t) - q'(t))\cdot W'(q(t))[\varphi(t) - q(t)]dt \right|&\leq& c_1\|\varphi' - q'\|_{L^2}\|\varphi - q\|_{L^2}\\
		&\leq&  c_1\|\varphi - q\|_{1.2}^2.
	\end{eqnarray*}
	It follows that for $\alpha_2>0$ sufficiently large one has 
	$$\mathcal H''_*(q)[\varphi - q,\varphi - q] \geq 0 \quad\mbox{for all} \ q, \varphi \in D(\Psi_*),$$
	which implies that $\mathcal H_*:\operatorname{int}(D(\Psi_*))\to\RR$ is convex. This together with the continuity of $\mathcal H_*$ implies the convexity of $\mathcal H_*$ on $D(\Psi_*)$.
	
	Thus, if $\alpha$ is sufficiently large, then the function
	$$D(\Psi_*)\ni q\mapsto \mathcal F_*(q)+\alpha \|q\|^2_{1,2}\in\RR$$
	is convex, which together with $\Psi_*$ being convex on $D(\Psi_*)$ implies the conclusion. \cqfd
	
	\begin{Lem}\label{boundps}
		If the conditions $(V_1)$ and $(W_1)$ are satisfied, then 
		the action $\mathcal I_*:H^1_{2\pi}\to (-\infty,+\infty]$  is bounded from below and satisfies $(wPS)_c$-condition with respect to $C_{2\pi}$ for any $\displaystyle c< -2\pi(l^*+\sup_{\RR^3}|W|).$
	\end{Lem}
	\proof First of all, notice that from $(V_1)$  it follows that $V$ is bounded. Moreover, from $(W_1)$ it follows that for any $q\in D(\Psi_*)$, we have  
	$$\left|\int_0^{2\pi}q'(t)\cdot W(q(t))dt\right|\leq 2\pi \sup_{\RR^3}|W|<\infty,$$
	so $\mathcal I_*$ is bounded from below on $H^1_{2\pi}.$ Next, consider $c< -2\pi(l^*+\sup_{\RR^3}|W|),$ and let $(q_n)\subset D(\Psi_*)$ be a $(PS)$-sequence at the level $c.$ Assume that $(\overline{q_n})$ is not bounded, and passing to a subsequence assume that $|\overline{q_n}|\to\infty.$ Consider
	the decomposition $q_n=\overline{q_n} + {\tilde q_n}$ and notice that $\|{\tilde q_n}\|_{\infty}\leq 2\pi$ for any $n\in\mathbb{N}.$ Thus, using $(V_1)$ we deduce that
	$$\int_0^{2\pi}V(q_n(t))dt\to 2\pi l^* \ \mbox{as} \ n\to \infty.$$
On the other hand, for all $n\in\mathbb{N},$ we have that
\begin{eqnarray*}
\mathcal I_*(q_n) &=& \int_0^{2\pi}[1-\sqrt{1-|q_n'|^2}]dt-\int_0^{2\pi}V(q_n)dt + \int_0^{2\pi}q_n'\cdot W(q_n)dt\\
		&\geq& -2\pi\sup_{\RR^3}|W| -\int_0^{2\pi}V(q_n)dt
\end{eqnarray*}
Letting $n\to\infty$ it follows that $c\geq  -2\pi(l^*+\sup_{\RR^3}|W|),$ a contradiction.
 Hence, $(\overline{q_n})$ is bounded
	and the conclusion follows using Proposition \ref{pslfe}.\cqfd
	
	\begin{Lem}\label{lemaroundzero}
		Assume additionally that the following condition holds true.
		\begin{itemize}
			\item $(V_2)$ There exist  $\lambda>0,$ and $r_{0}>0$ such that $V(q)\geq\lambda |q|^2$ for all $q\in\RR^3$ with $|q|\leq r_{0}.$
		\end{itemize}
			%\item $(W_2)$ There exist $\nu>2,$ $s_{0}>0,$ and $\delta>0$ such that $|W(q)|\leq\delta |q|^{\nu}$   
			%for all $q\in\RR^3$ with $|q|\leq s_{0}.$
		Consider $m\in\mathbb{N}$ and 
		$$Z_m=\left\{\sum_{j=1}^m [\cos(jt)a_j + \sin(jt)b_j]: a_j, b_j\in\RR^3, 1\leq j\leq m\right\},$$
		and $$D=\{q\in Z_m: \|q\|_{1,\infty}<r\},$$ where $r<\min(r_{0}, 1)$. Then, $Z_m$ is a $(6m)$-dimensional  invariant subspace of $H^1_{2\pi}$ such that $Z_m\cap\operatorname{Fix}(S^1)=\{0\},$ and $D$ is an open bounded invariant neighbourhood of $0$ in $Z_m.$ Moreover, there exists $\Lambda_m(l^*, W)>0$ such that if $\lambda\geq\Lambda_m,$ then 
$$\sup_{\partial D}\mathcal I_*< -2\pi(l^*+\sup_{\RR^3}|W|).$$
	\end{Lem}
	\proof One has that $Z_m$ is invariant due to the  trigonometric formulas
\begin{eqnarray*}
\cos(x+y)=\cos(x)\cos(y)-\sin(x)\sin(y), \\
 \sin(x+y)=\sin(x)\cos(y)+\cos(x)\sin(y).
\end{eqnarray*}
	Thus, all we have to prove is that $\mathcal I_*(q)<-2\pi(l^*+\sup_{\RR^3}|W|)$ for all $q\in\partial D$ and for all $\lambda$ large enough. Using that $r\leq 1,$ one has $\overline{D}\subset D(\Psi_*),$ and using that
	\begin{eqnarray*}
		1-\sqrt{1-s^2}\leq s^2 \quad\mbox{for all} \ s\in [0,1],
	\end{eqnarray*}
	and 
	\begin{eqnarray*}
		m^2\int_0^{2\pi}|q(t)|^2dt\geq \int_0^{2\pi}|q'(t)|^2dt \quad\mbox{for all} \ q\in Z_m,
	\end{eqnarray*}
	from $(V_2)$ and $(W_1)$ we deduce that for all $q\in \overline{D},$
	\begin{eqnarray*}
		\mathcal I_*(q) &=& \int_0^{2\pi}[1-\sqrt{1-|q'(t)|^2}]dt-\int_0^{2\pi}V(q(t))dt + \int_0^{2\pi}q'(t)\cdot W(q(t))dt\\
		&\leq& (m^2 - \lambda)\int_0^{2\pi}|q(t)|^2dt + 2\pi \sup_{\RR^3} |W|.
	\end{eqnarray*}
	Let $\gamma_m>0$ be such that
	$$||q||_{L^{2}}\geq \gamma_m ||q||_{1,\infty} \quad\mbox{for all} \ q\in Z_m.$$
	It follows that for all $q\in\partial D,$ one has that
	\begin{eqnarray*}
		\mathcal I_*(q)\leq (m^2 - \lambda)2\pi \gamma_mr^2 + 2\pi\sup_{\RR^3}|W|,
	\end{eqnarray*}
	and the conclusion follows.\cqfd

\begin{Rem}
{\em The constant $\Lambda_m$ quantifies the link between the behaviour of the electric potential $V$ around the origin and the behaviour of the electric potential $V$ at infinity together with the magnetic potential $W$.  }
\end{Rem}
	
	\begin{Lem}\label{lemfix}
		Consider the set 
		$$\Omega_*=\{q\in H^1_{2\pi}: \mathcal I_*(q)<0\}.$$
		Then, we have that
		$$\operatorname{Fix}(S^1)\cap \Omega_* \cap K(\mathcal I_*) = \emptyset.$$
	\end{Lem}
	\proof Assume that there exists $q\in\RR^3=\operatorname{Fix}(S^1)$ such that $q$ is a critical point of $\mathcal I_*$ and $\mathcal I_*(q)<0.$
	From Theorem \ref{criticsol} it follows that $q$ is a solution of the Lorentz force equation, that is $V'(q)=0,$ a contradiction. The proof is completed.\cqfd
	
	Our main application goes as follows.
	
	\begin{The}\label{theappl}
		Suppose that the assumptions $(V_{1, 2})$  and $(W_{1})$ hold true. For any $m\in\mathbb{N}$ there exists $\Lambda_m>0$ such that if $\lambda$ in $(V_2)$ satisfies $\lambda\geq\Lambda_m,$ then $\mathcal I_*$ has at least $3m$ critical orbits at negative levels, which are $2\pi$-periodic solutions of the Lorentz force equation.
	\end{The}
	\proof The result follows from  Theorem \ref{criticsol},  Corollary \ref{maincorabstract}, and Lemmas \ref{ELconvlem}, \ref{boundps}, \ref{lemaroundzero}, \ref{lemfix}.
\cqfd

\begin{Exa}
{\em 
For any $\lambda>0$ one has that $$\lim_{x\to 0}x^{-2}\arctan (\lambda x^2)=\lambda, \quad \lim_{|x|\to\infty}\arctan (\lambda x^2)=\frac{\pi}{2}.$$
Consider the electric potential  $V$ given by
$$V:\RR^3\to\RR, \qquad V(q)=\arctan (\lambda |q|^2).$$
One has that $V$ satisfies $(V_{1, 2})$ with $l^*=\frac{\pi}{2}$ and $V(q)\geq \frac{\lambda}{2}|q|^2$ around  zero. Assume also that the magnetic potential $W$ satisfies $(W_{1}).$ Let $m\in\mathbb{N}$ be fixed. From the above theorem it follows that there exists $\Lambda_m>0$ such that the Lorentz force equation has at least $3m$ $2\pi$-periodic solutions for any $\lambda$ such that $\frac{\lambda}{2}\geq \Lambda_m.$ In particular, if $\lambda\to\infty,$ then the number of  $2\pi$-periodic solutions of the Lorentz force equation goes to infinity.}
\end{Exa}

\begin{Rem}
{\em It is well known that without the normalization we used in this paper the relativistically correct equation of the motion of a charged particle is
\begin{eqnarray*}
\left(\frac{q'}{\sqrt{1-|q'|^2/c^2}}\right)' =\frac{\beta}{m_0}(E(t,q)+q'\times B(t,q)),
\end{eqnarray*}
where $c$ is the speed of light, $m_0$ is the rest mass, and $\beta\in \RR$ is the charge of the  particle. Notice that $m_0$ and $\beta$ are prescribed. Thus, to apply  our result, if we want a positively charged particle, our electric potential $V$ must be attractive, i.e. $V>0,$ and if we want a negatively charged particle, our electric potential $V$ must be repulsive, i.e. $V<0$.
}
\end{Rem}

\begin{flushright}
\begin{tabular}{@{}l@{}}
		University of Bucharest\\
		Faculty of Mathematics\\
		14 Academiei Street\\
		70109 Bucharest, Romania\\
		and\\
		Institute of Mathematics ``Simion Stoilow''\\
		Romanian Academy\\
		21 Calea Grivitei, Bucharest, Romania\\
		cbereanu@imar.ro\\
		\\
		Babe\c{s}-Bolyai University\\
		Faculty of Mathematics and Computer Science\\
		1 Mihail Kog\u{a}lniceanu \\
		400084 Cluj-Napoca, Romania\\
		alexandru.pirvuceanu@ubbcluj.ro
\end{tabular}
\end{flushright}
\end{document}